\documentclass[11pt]{amsart}
\usepackage{txfonts}
\usepackage{amsmath,amsthm,verbatim}
\usepackage{amscd,amssymb}
\usepackage[all]{xy} 
\usepackage{graphicx,calrsfs} 
\usepackage[colorlinks,plainpages,backref,urlcolor=blue]{hyperref}
\usepackage{breakurl}
\usepackage{bbm}
\usepackage{bbm}
\usepackage{mathrsfs}
\usepackage{xstring}
\usepackage{tikz}
\usetikzlibrary{matrix,arrows,decorations.pathmorphing,calc,shapes}

\newtheorem{theorem}{Theorem}[section]
\newtheorem{corollary}[theorem]{Corollary}
\newtheorem{lemma}[theorem]{Lemma}
\newtheorem{prop}[theorem]{Proposition}

\theoremstyle{definition}
\newtheorem{definition}[theorem]{Definition}
\newtheorem{example}[theorem]{Example}

\newtheorem*{ack}{Acknowledgments}

\DeclareMathOperator{\im}{im}

\DeclareMathOperator{\id}{id}

\DeclareMathOperator{\ab}{ab}

\DeclareMathOperator{\Prim}{Prim}
\DeclareMathOperator{\gr}{gr}

\DeclareMathOperator{\GL}{GL}
\DeclareMathOperator{\Aut}{Aut}
\DeclareMathOperator{\Inn}{Inn}

\DeclareMathOperator{\IA}{IA}

\newcommand{\N}{\mathbb{N}}
\newcommand{\R}{\mathbb{R}}
\newcommand{\Q}{\mathbb{Q}}
\newcommand{\C}{\mathbb{C}}
\newcommand{\Z}{\mathbb{Z}}

\renewcommand{\k}{\Bbbk}
\newcommand{\g}{\mathfrak{g}}
\newcommand{\h}{\mathfrak{h}}

\newcommand{\fh}{\mathfrak{h}}
\newcommand{\fg}{\mathfrak{g}}

\newcommand{\fM}{\mathfrak{M}}
\newcommand{\fm}{\mathfrak{m}}

\newcommand{\fr}{\mathfrak{r}}

\newcommand{\Lie}{\mathfrak{lie}}
\newcommand{\lie}{\mathfrak{lie}}

\newcommand{\cF}{\mathcal{F}}
\newcommand{\cG}{\mathcal{G}}

\newcommand{\wB}{wB}
\newcommand{\wP}{wP}

\DeclareMathAlphabet{\pazocal}{OMS}{zplm}{m}{n}

\newcommand{\surj}{\twoheadrightarrow}
\newcommand{\inj}{\hookrightarrow}

\newcommand{\abs}[1]{\left| #1 \right|}
\newcommand{\ncom}[1]{\langle\!\langle #1 \rangle\!\rangle}
\newcommand{\com}[1]{[\![ #1 ]\!]}
\newcommand\isom{\xrightarrow{
 \,\smash{\raisebox{-0.6ex}{\ensuremath{\scriptstyle\simeq}}}\,}}

\newcommand{\hdot}{{\raise1.2pt\hbox to0.3em{\Large $\dot$}}}

\begin{document}

\title[Taylor expansions of groups and filtered-formality]%
{Taylor expansions of groups and filtered-formality}

\author[Alexander~I.~Suciu]{Alexander~I.~Suciu$^1$}
\address{Department of Mathematics,
Northeastern University,
Boston, MA 02115, USA}
\email{\href{mailto:a.suciu@northeastern.edu}{a.suciu@northeastern.edu}}
\urladdr{\href{http://web.northeastern.edu/suciu/}%
{http://web.northeastern.edu/suciu/}}
\thanks{$^1$Supported in part by the Simons Foundation 
collaboration grant for mathematicians 354156.}

\author{He Wang}
\address{Department of Mathematics,
Northeastern University,
Boston, MA 02115, USA}
\email{\href{mailto:wanghemath@gmail.com}%
{wanghemath@gmail.com}, \href{mailto:he.wang@northeastern.edu}%
{he.wang@northeastern.edu}}



\dedicatory{To the memory of \c{S}tefan Papadima, 1953--2018}

\begin{abstract} 
Let $G$ be a finitely generated group, and let $\k{G}$ be its group algebra 
over a field of characteristic $0$.  A Taylor expansion is a certain type of map from 
$G$ to the degree completion of the associated graded algebra of $\k{G}$ 
which generalizes the Magnus expansion of a free group. 
The group $G$ is said to be filtered-formal if its Malcev Lie 
algebra is isomorphic to the degree completion of its associated 
graded Lie algebra.   We show that $G$ is filtered-formal if and 
only if it admits a Taylor expansion, and derive some consequences.
\end{abstract}

\maketitle
\setcounter{tocdepth}{1}
\tableofcontents

\section{Introduction}
\label{sect:intro}

\subsection{Expansions of groups} 
\label{intro:exp}

Group expansions were first introduced by Magnus in \cite{Magnus35}, 
in order to show that finitely generated free groups are residually nilpotent. 
This technique has been generalized and used in many ways. 
For instance, the exponential expansion of a free group was used 
to give a presentation for the Malcev Lie algebra 
of a finitely presented group by Papadima \cite{Papadima95} 
and Massuyeau \cite{Massuyeau12}. 
Expansions of pure braid groups and their applications in knot theory 
have been studied since the 1980s by several authors,
see for instance Kohno's papers \cite{Kohno85, Kohno88, Kohno16}.
X.-S. Lin studied in  \cite{Linxiaosong97} expansions of fundamental 
groups of smooth manifolds, using K.T. Chen's theory \cite{Chen77} of formal power 
series connections and their induced monodromy representations.  
More generally, Bar-Natan has explored in \cite{Bar-Natan16}  
the Taylor expansions of an arbitrary ring.

Let $G$ be a finitely generated group, and fix a coefficient 
field $\k$ of characteristic zero. We let ${\gr}(\k{G})$ be the associated 
graded algebra of $\k{G}$ with respect to the filtration by powers of the 
augmentation ideal, and we let 
$\widehat{\gr}(\k{G})$ be the degree completion of this algebra.
Developing an idea from \cite{Bar-Natan16}, we say that a map 
$E\colon G\to \widehat{\gr}(\k{G})$ is a multiplicative expansion 
of $G$ if the induced algebra morphism,  
$\bar{E}\colon \k{G}\to \widehat{\gr}(\k{G})$,  
is filtration-preserving and induces the identity at the associated graded level.
Such a map $E$ is called a \emph{Taylor expansion}\/  if it sends 
each element of $G$ to a group-like element of the Hopf algebra 
$\widehat{\gr}(\k{G})$.

\subsection{Expansions  and filtered-formality} 
\label{intro:exp formal}

Once again, let $G$ be a finitely generated group. 
The concept of filtered-formality relates an object from rational homotopy 
theory to a group-theoretic object. The first object is the Malcev Lie 
algebra $\fm(G,\k)$, defined by Quillen \cite{Quillen69} 
as the set of primitive elements of the 
$I$-adic completion of the group algebra of $G$, where $I$ is the augmentation 
ideal of  $\k{G}$.  This Lie algebra comes endowed with a (complete) filtration 
induced from the natural filtration on  $\widehat{\k{G}}$,  
and is isomorphic to the dual of Sullivan's $1$-minimal model of
a  $K(G,1)$ space.
The second object is the graded Lie algebra $\gr(G,\k)$, defined by taking 
the direct sum of the successive quotients of the lower central series of $G$, 
tensored with $\k$. As shown by Quillen in \cite{Quillen68}, the associated 
graded algebra ${\gr}(\k{G})$ is isomorphic to the universal enveloping 
algebra of $\gr(G,\k)$. 

The group $G$ is called {\em filtered-formal}\/ if its 
Malcev Lie algebra, $\fm(G,\k)$, is isomorphic
to $\widehat{\gr}(G;\k)$, the degree completion of its associated 
graded Lie algebra, as filtered Lie algebras. 
If, in addition, the graded Lie algebra $\gr(G;\k)$ is 
quadratic, the group $G$ is said to be {\em $1$-formal}.
For more details on these notions we refer to 
\cite{PS04-imrn, Papadima-Suciu09, SW-formal} 
and references therein.

The following result, which elucidates the relationship 
between Taylor expansions and formality properties,
is a combination of Theorem \ref{thm:expansionFiltered} and 
Corollary  \ref{cor:expansionFormal}.

\begin{theorem}
\label{thm:intro expansion}
Let $G$ be a finitely generated group.  Then:
\begin{enumerate}
\item \label{tt1}
$G$  is filtered-formal if and only if 
$G$ has a Taylor expansion $G\to \widehat{\gr}(\k{G})$.
\item \label{tt2} $G$ is $1$-formal if and only if
$G$ has a Taylor expansion  and $\gr(\k{G})$ is a quadratic algebra.
\end{enumerate}
\end{theorem}

Combining this theorem with our results on filtered-formality from \cite{SW-formal}, 
we conclude that the following propagation property of Taylor expansions holds. 
This is a combination of Theorems \ref{thm:Taylorpropagation} and \ref{thm:nilp-taylor}.

\begin{prop}
\label{prop:intro propagation}
The existence of a Taylor expansion  is preserved under 
field extensions, and taking finite products and coproducts, 
split injections, nilpotent quotients or solvable quotients of groups. 
\end{prop}

In particular, if a finitely generated group $G$ 
has a Taylor expansion over $\C$, then it also has a Taylor expansion 
over $\Q$.

 \subsection{Residual properties and Taylor expansions} 
 \label{intro:rtfn-intro}

A group $G$ is said to be \emph{residually torsion-free nilpotent}\/ 
if any non-trivial element of $G$ can be detected in a 
torsion-free nilpotent quotient.  If $G$ is finitely generated, 
this condition is equivalent to the injectivity of the canonical 
map to the Malcev group completion, $\kappa\colon G\to \fM(G,\k)$. 
An expansion $\bar{E}\colon\k G\to \widehat{\gr}(\k G)$ is said to be 
{\em faithful}\/ if the map $E\colon G\to \widehat{\gr}(\k{G})$ is injective.  

The next proposition relates the property 
of being residually torsion-free nilpotent to the existence of 
a faithful Taylor expansion.  
 \begin{prop}
\label{prop:fftaylor-intro}
A finitely generated group $G$ has a faithful Taylor expansion 
if and only if $G$ is residually torsion-free nilpotent and filtered-formal. 
\end{prop}

The work of Magnus \cite{Magnus35,Magnus-K-S} shows that all the 
free groups $F_n$ are residually torsion-free nilpotent (RTFN). 
Furthermore, as shown by Hain \cite{Hain97} and Berceanu--Papadima 
\cite{Berceanu-Papadima09}, the Torelli groups  
$\IA_n=\ker(\Aut(F_n)\to \Aut((F_n)_{\ab}))$ 
are also RTFN.   
Consequently, all the subgroups of $\IA_n$, for instance, 
the pure braid group $P_n$, the McCool group $wP_n$, and the 
upper McCool group $wP_n^+$, inherit this property.

Let $\Pi_n$ be the direct product of the free groups $F_{1}, \dots, F_{n-1}$. 
The graded Lie algebras $\gr(P_n,\k)$, $\gr(\Pi_n,\k)$ and $\gr(\wP^+_n,\k)$
are isomorphic as vector spaces. Hence, their universal enveloping algebras,
which are domains for the Taylor expansions of $P_n$, $\Pi_n$, and $\wP^n$,
are also isomorphic as vector spaces. The next proposition shows that
they are not isomorphic as algebras.  
 
\begin{prop}
\label{prop:theta-intro}
For each $n\geq 4$, the graded Lie algebras $\gr(P_n,\k)$, $\gr(\Pi_n,\k)$, and 
$\gr(\wP^+_n,\k)$ are pairwise non-isomorphic.  
\end{prop}
  
 \subsection{Braid-like groups and further directions} 
\label{subsec:motivation} 
Explicit Taylor expansions have been constructed for 
several classes of filtered-formal groups, including  
finitely generated free groups, free abelian groups, surface groups,
the pure braid groups, and the McCool groups. 

When $G$ is the fundamental group of a smooth manifold $M$,
an important construction for a Taylor expansion arises from Chen's theory of 
formal power series connections and their induced monodromy representations.  
Using this technique, Kohno \cite{Kohno88, Kohno16} gave explicit 
Taylor expansions for the pure braid groups $P_n$.  Using a completely 
different approach, Papadima constructed in \cite{Papadima02} 
 {\em integral}\/ Taylor expansions for the braid groups $B_n$. 
In another direction,  
Hain studied expansions for link groups \cite{Hain85}, fundamental groups 
of algebraic varieties \cite{Hain86}, and the Torelli groups \cite{Hain97}, 
while Lin \cite{Linxiaosong97} further investigated the relationship between 
expansions and link invariants, including Vassiliev invariants, Milnor's link 
variants and the Kontsevich integral.  

There is also a strong interplay between Taylor expansions of 
the pure braid groups and the finite-type (or Vassiliev) 
invariants in knot theory.  In this context, the relevant formal power series 
connection is a version of the Knizhnik--Zamolodchikov connection. 
The Taylor expansions of the groups constructed from Chen's theory of 
formal power series connections yield finite-type invariants for pure braids, 
and provide a prototype for the Kontsevich integral for knots.
For more on all of this, we refer the reader to \cite{Habegger-Massbaum, 
Mostovoy-Willerton, Papadima02,Linxiaosong97, Bar-Natan95, Bar-Natan16}.
 
\section{Hopf algebras and expansions of groups}
\label{sec:expansion}

\subsection{Group algebras, completions, and associated graded algebras}
\label{subsec:grkg}

Let  $G$ be a finitely generated group, and let $\k{G}$ 
be its group algebra over a field $\k$.
Let $\varepsilon\colon \k{G}\to \k$ be the augmentation 
homomorphism, defined by $\varepsilon(g)=1$ for all $g\in G$. 
The powers of the augmentation ideal, $I=\ker(\varepsilon)$, 
define the $I$-adic filtration on the group algebra, $\{I^k\}_{k\ge 0}$. 
This filtration is multiplicative, in the sense that 
$I^k\cdot I^{\ell}\subset I^{k+\ell}$. 
The corresponding completion, 
\begin{equation}
\label{eq:hat-kg}
\widehat{\k{G}}=\varprojlim_k \k{G}/I^k,
\end{equation} 
comes equipped with the inverse limit filtration,
$\{\widehat{I^k}\}_{k\ge 0}$.  The multiplication in $\k{G}$ 
extends to a multiplication in $\widehat{\k{G}}$, compatible 
with this filtration. 

On the other hand, the associated graded group, 
\begin{equation}
\label{eq:grkG}
{\gr}(\k{G})=\bigoplus_{k\geq 0}I^k/I^{k+1},
\end{equation}
is a graded algebra, with multiplication inherited from the 
product in $\k{G}$. 
This algebra comes endowed 
with the degree filtration, 
$\cF_k(\gr(\k{G}))=\bigoplus_{j\geq k}I^j/I^{j+1}$. 
The completion of $\gr(\k{G})$ with respect to this filtration, 
\begin{equation}
\label{eq:completegr}
\widehat{\gr}(\k{G})=\prod_{k\geq 0}I^k/I^{k+1},
\end{equation}
comes endowed with the inverse limit filtration, 
\begin{equation}
\label{eq:completefil}
\widehat{\cF}_k(\widehat{\gr}(\k{G}))=\prod_{j\geq k}I^j/I^{j+1}.
\end{equation}
The associated graded algebra of $\widehat{\gr}(\k{G})$ 
is canonically identified with $\gr(\k{G})$.

For example, if $G=F_n$ is a free group of rank $n\ge 1$, 
then $\gr(\k{G})$ is the tensor $\k$-algebra on $n$ generators 
$t_i$ while the completion $\widehat{\gr}(\k{G})$ is the power series 
ring in $n$ non-commuting variables $x_i=t_i-1$.

\subsection{Hopf algebras}
\label{subsec:hopf}

A {\em Hopf algebra}\/ is an associative and coassociative bialgebra 
over a field $\k$, with multiplication $\nabla\colon A\otimes A\to A$, 
comultiplication $\Delta\colon A\to A\otimes A$, unit 
$\eta\colon \k\to A$, and counit $\varepsilon \colon A\to \k$, 
endowed with a $\k$-linear map $T\colon A\to A$ 
(called the antipode), such that the following diagram commutes:
\[
\xymatrixrowsep{24pt}
\xymatrixcolsep{14pt}
\xymatrix{
& A\otimes A \ar^{T\otimes \id}[rr] && A\otimes A \ar^{\nabla}[dr]\\
A \ar^{\varepsilon}[rr]  \ar^{\Delta}[ur] \ar_{\Delta}[dr]&& \k \ar^{\eta}[rr]  && A\\
& A\otimes A \ar^{\id\otimes T}[rr] && A\otimes A \ar_{\nabla}[ur]
}
\]
 
An element $x\in A$  is called {\em group-like}\/ if 
$\Delta(x)=x\otimes x$, and it is called {\em primitive}\/ if 
$\Delta x=x {\otimes} 1+ 1{\otimes} x$.
The set of group-like elements of $A$ 
form a group, with multiplication inherited from $A$ and 
inverse given by the antipode, 
while the set of primitive elements of $A$ form a Lie algebra, 
with Lie bracket $[x,y]=\nabla(x,y)-\nabla(y,x)$.

For instance, if $\fg$ is a Lie algebra, then its universal enveloping algebra, 
$U(\fg)$, is a Hopf algebra, with $\Delta x=x {\otimes} 1+ 1{\otimes} x$, 
$\varepsilon(x)=0$, and $T(x)=-x$ for all $x\in \fg$. 
By construction, the set of primitive elements in $U(\fg)$ 
coincides with $\fg$.  Suppose now that $\fg\cong \k^n$, 
with Lie bracket equal to $0$. We may then identify $U(\g)$ 
with the polynomial ring $\k[x_1,\dots , x_n]$.  Likewise, 
if $\widehat{U}(\fg)$ denotes the completion of $U(\fg)$ 
with respect to the filtration by powers of the augmentation ideal 
$J=\ker(\varepsilon)$,  we may then identify $\widehat{U}(\fg)$ 
with the power series ring $\k\com{x_1,\dots , x_n}$. 

From now on, we will assume that $\k$ is a field of characteristic $0$. 
As is well-known, the group algebra $\k{G}$ of a group $G$ is a Hopf algebra, 
with comultiplication $\Delta\colon \k{G}\to \k{G}\otimes \k{G}$ 
given by $\Delta(g)=g\otimes g$ for $g\in G$, counit $\varepsilon \colon \k{G}\to \k$ 
the augmentation map, and antipode $T\colon \k{G}\to \k{G}$ 
given by $T(g)=g^{-1}$. 
In \cite{Quillen69}, Quillen showed that the  $I$-adic completion 
of the group algebra,  
$\widehat{\k{G}}$, is a complete Hopf algebra, with comultiplication map 
\begin{equation}
\label{eq:hatdelta}
\xymatrixcolsep{16pt}
\xymatrix{\widehat{\Delta} \colon \widehat{\k{G}}\ar[r]& \widehat{\k{G}}\, 
\hat{\otimes}\,  \widehat{\k{G}}
}.
\end{equation}
where $\hat{\otimes}$ denotes the completed tensor product, defined 
in this case as  $\widehat{\k{G}}\, \hat{\otimes}\,  \widehat{\k{G}}=
\varprojlim_k \k{G}/I^k \otimes \k{G}/I^k$. 
Identifying the associated graded algebra $\gr\big(\k{G}\otimes \k{G}\big)$ 
with $\gr(\k{G})\otimes \gr(\k{G})$, 
we see that the degree completion $\widehat{\gr}(\k{G})$ is also 
a complete Hopf algebra, with comultiplication map 
\begin{equation}
\label{eq:bardelta}
\xymatrixcolsep{18pt}
\xymatrix{\bar{\Delta}:=\widehat{\gr}(\Delta) \colon\widehat{\gr}(\k{G})\ar[r]
&\widehat{\gr}(\k{G})\, \hat{\otimes}\,  \widehat{\gr}(\k{G})}.
\end{equation}

\subsection{Multiplicative expansions and Taylor expansions}
\label{subsec:exp}

Given a map $f\colon G\to R$, where $R$ is a ring, we will 
denote by $\bar{f} \colon \k{G}\to R$ its linear extension to 
the group algebra. 

\begin{definition}
\label{def:exp}
A \emph{\textup{(}multiplicative\textup{)} expansion}\/
of a group $G$ is a map 
\begin{equation}
\label{eq:expansion}
\xymatrixcolsep{16pt}
\xymatrix{
E\colon G \ar[r]& \widehat{\gr}(\k G)
}
\end{equation}
such that the linear extension $\bar{E}\colon\k G\to \widehat{\gr}(\k G)$  
is a filtration-preserving algebra morphism with the property that $\gr(\bar{E})=\id$. 
Furthermore, we say that the expansion $E$ is {\em faithful}\/ if $E$ is injective. 
\end{definition}
 
Alternatively, an expansion of $G$ is a (multiplicative) monoid 
map $E\colon G \to \widehat{\gr}(\k{G})$ such that the following 
property holds:  If $f\in I^k\setminus I^{k+1}$, then $\bar{E}(f)$ 
starts with $[f]\in I^k/I^{k+1}$,  that is, $\bar{E}(f)=(0,\dots,0,[f],*,*,\dots)$.  

Following Bar-Natan \cite{Bar-Natan16}, we make the following definition.

\begin{definition}
\label{def:taylor} 
An expansion $E\colon G \to \widehat{\gr}(\k G)$ is 
called a \emph{Taylor expansion}\/ 
(or, a \emph{group-like}\/ expansion)  if it sends all elements of 
$G$ to group-like elements of $\widehat{\gr}(\k G)$, 
that is, 
\begin{equation}
\label{eq:DeltaBar}
\xymatrix{\bar\Delta (E(g))=E(g) \hat{\otimes} E(g)
}
\end{equation} 
for all $g\in G$.  
\end{definition}

Equivalently, an expansion $E$ is a Taylor expansion if it is \emph{co-multiplicative}, 
i.e., the following diagram commutes:
\begin{equation}
\label{eq:grouplike}
\begin{gathered}
\xymatrix{
\k{G} \ar[d]^{{\bar{E}}}  \ar[r]^(.45){{\Delta}} &\k{G}\otimes \k{G} \ar[d]^{{\bar{E}} {\otimes}{\bar{E}}}\\
\widehat{\gr}(\k G)  \ar[r]^(.4){\bar\Delta} &\widehat{\gr}(\k G) \hat{\otimes}\widehat{\gr}(\k G)   . 
}
\end{gathered}
\end{equation}

\begin{prop}
\label{prop:Taylor}
A Taylor expansion $E\colon G\to \widehat{\gr}(\k{G})$ induces a 
filtration-preserving isomorphism of complete Hopf algebras,  
$\widehat{E}\colon \widehat{\k G}\to \widehat{\gr}(\k G)$, such that 
$\gr(\widehat{E})$ is the identity on $\gr(\k{G})$. 
\end{prop}

\begin{proof}
As in the above definition, the expansion $E$ induces a filtration-preserving algebra morphism,  
$\bar{E}\colon \k G\to \widehat{\gr}(\k G)$.  Applying the $I$-adic completion functor, 
we obtain an algebra morphism, $\widehat{E} \colon \widehat{\k G}\to \widehat{\gr}(\k G)$. 
By the above discussion, the expansion $E$ is group-like if and only if 
$\widehat{E}$ is co-multiplicative.   Applying the completion functor to 
diagram \eqref{eq:grouplike} yields another commuting diagram, 
\begin{equation}
\begin{gathered}
\xymatrix{
\widehat{\k G} \ar[d]^{\widehat{E}}  \ar[r]^(.45){\widehat{\Delta}} 
&\widehat{\k G}\hat{\otimes}\widehat{\k G} 
\ar[d]^{\widehat{E}\hat{\otimes}\widehat{E}}\\
\widehat{\gr}(\k G)  \ar[r]^(.4){\bar{\Delta}} 
&\widehat{\gr}(\k G) \hat{\otimes}\widehat{\gr}(\k G).
}
\end{gathered}
\end{equation}

Since $\bar{E}$ is filtration-preserving and $\gr(\bar{E})=\id$, this implies 
that the Hopf algebra morphism $\widehat{E}$ 
preserves filtrations and that $\gr(\widehat{E})=\id$.  
By induction on $k$, all induced maps 
$\widehat{\k G}/\widehat{I^k} \to\widehat{\gr}(\k G)/\widehat{\cF}_k$ 
are isomorphisms,
where $\widehat{\cF}_k$ is the filtration from display \eqref{eq:completefil}. 
It follows from the next lemma that $\widehat{E}$ is an isomorphism. 
\end{proof}

\begin{lemma}
\label{lem:completefiltration}
Let  $f\colon  A\to B$ be a morphism of 
filtered, complete, and separated algebras. 
If $\gr(f)\colon  \gr^{\cF}(A)\to \gr^{\cG}(B)$ 
is an isomorphism, then $f$ is also an isomorphism.
\end{lemma}

\begin{proof}
 By assumption, the homomorphisms 
$\gr_k(f)\colon \cF_{k}A/\cF_{k+1}A \to \cG_k B/\cG_{k+1} B$ 
are isomorphisms, for all $k\ge 1$.  An easy induction on $k$ 
shows that all maps $f_k\colon A/\cF_{k+1}A \to B/\cG_{k+1}B $ 
are isomorphisms. Therefore, the map 
$\hat{f} \colon \widehat{A} \to \widehat{A}$ is 
an isomorphism. On the other hand, both $A$ and $B$ are 
complete and separated, and so $A=\widehat{A}$ 
and $B=\widehat{B}$.  Hence $f=\hat{f}$, and we are done.
\end{proof}

\subsection{On the existence of Taylor expansions}
\label{subsec:alt-taylor}

As we shall see, not all finitely generated groups admit a Taylor
expansion.  We conclude this section with an if-and-only-if criterion 
for the existence of a such expansion. 

\begin{prop}
\label{prop:TaylorExp}
A filtration-preserving isomorphism of complete Hopf algebras,  
$\phi\colon \widehat{\k G}\to \widehat{\gr}(\k G)$, 
induces a Taylor expansion $E\colon G\to \widehat{\gr}(\k{G})$.
\end{prop}

\begin{proof}
The isomorphism $\phi$ induces 
a filtration-preserving isomorphism of complete Hopf algebras, 
$\tilde\phi:=(\widehat{\gr}(\phi))^{-1}\circ \phi$, from $\widehat{\k{G}}$ to 
$\widehat{\gr}(\k G)$,  such that $\gr(\tilde\phi)=\id$.
Let $E\colon G\to \widehat{\gr}(\k G)$ be the composite 
\begin{equation}
\label{eq:gkg}
\xymatrixcolsep{20pt}
\xymatrix{
 G\ar@{^{(}->}[r] & \k{G}\ar^{\jmath}[r]  & \widehat{\k{G}}\ar[r]^(.4){\tilde\phi}
 &  \widehat{\gr}(\k G).
}
\end{equation}

Since both $\tilde\phi$ and $\jmath$ are morphisms of Hopf algebras, and 
since the inclusion $G\inj \k{G}$ is a monoid map sending $G$ to 
the group-like elements of $\k{G}$, the composite 
$E$ is also a monoid map. 
It is clear that $\widehat{E}=\tilde\phi$ and $\bar{E}=\tilde\phi\circ \jmath$. 
Since both $\tilde\phi$ and $\jmath$ are filtration-preserving, 
and $\gr(\jmath)=\gr(\widehat{E})=\id$, we infer that $\bar{E}$ is filtration-preserving 
and $\gr(\bar{E})=\id$.
Finally, by construction,  $E$ is a group-like expansion.
 \end{proof}

Propositions \ref{prop:Taylor} and \ref{prop:TaylorExp} can be summarized as follows. 
 
\begin{theorem}
\label{thm:TaylorHopf}
The assignment $E\leadsto \widehat{E}$ establishes a 
one-to-one correspondence between Taylor expansions 
$G\to \widehat{\gr}(\k{G})$ and filtration-preserving isomorphisms 
of complete Hopf algebras $\widehat{\k G}\to \widehat{\gr}(\k G)$ for which  
the associated graded morphism is the identity on $\gr(\k{G})$. 
\end{theorem}

This theorem generalizes a result of Massuyeau (\cite[Proposition~2.10]{Massuyeau12}), 
from finitely generated free groups to arbitrary finitely generated groups.  
Proposition \ref{prop:TaylorExp} and  Theorem \ref{thm:TaylorHopf} have 
as an immediate corollary the aforementioned 
criterion for the existence of a Taylor expansion.

\begin{corollary}
\label{cor:te}
A finitely generated group $G$ has a Taylor expansion if and only if 
there is an isomorphism of filtered Hopf algebras,  
$\widehat{\k G}\cong\widehat{\gr}(\k G)$.
\end{corollary}

\section{Chen iterated integrals and Taylor expansions}
\label{sec:chenTaylor}

\subsection{Chen iterated integrals}
\label{subsec:chen}
In \cite{Chen73,Chen77}, Chen developed a theory of formal power series 
connections and iterated integrals on smooth manifolds. 
His original motivation was to describe the homology 
of the loop space of a smooth manifold $M$ 
in terms of the differential graded algebra formed by tensoring the de 
Rham algebra $\Omega_{\textrm{DR}}(M)$ with the tensor algebra 
on the vector space $H_{>0}(M,\R)$, completed with respect 
to the powers of the augmentation ideal.  As summarized below, 
Chen's theory leads to monodromy representations of the fundamental 
group of $M$  (see also Lin \cite{Linxiaosong97} and Kohno \cite{Kohno16} 
for further details).

For simplicity, we will assume the manifold $M$ has the homotopy type 
of a connected, finite-type CW-complex. 
Upon choosing a basis $\mathbf{X}=\{X_i\}_i$ for $\widetilde{H}_*(M,\k)$, 
we may identify the algebra 
$\Omega_{\textrm{DR}}(M) \otimes_{\k} \widehat{T}(\widetilde{H}_*(M,\k))$ 
with $\Omega_{\textrm{DR}}(M)\ncom{\mathbf{X}}$. (Here, $\k=\R$ or $\C$.) 
A {\em formal power series connection}\/ 
on $M$ is an element $\omega\in \Omega_{\textrm{DR}}(M)\ncom{\mathbf{X}}$.   
We may write such an element (which may also be viewed as a usual connection 
on the trivial bundle $M\times \k\ncom{\mathbf{X}}$) as 
\begin{equation}
\label{eq:connection}
\omega=
\sum w_iX_i + \cdots + \sum w_{i_1\ldots i_r}X_{i_1}\cdots X_{i_r}+\cdots,
\end{equation}
where the coefficients are smooth forms of positive degree on $M$.  
A connection $\omega$ as above is said to be {\em flat}\/ if it satisfies 
the Maurer--Cartan equation,  $d\omega-\omega\wedge \omega=0$.

For a homology class $X\in \widetilde{H}_{p}(M,\k)$ 
we set $\deg X= p-1$; more generally, we set
$\deg (X_{i_1}\cdots X_{i_r}):=\deg X_{i_1}+\cdots +\deg X_{i_r}$.
We denote by $\omega_0\in \Omega^1_{\textrm{DR}}(M)\otimes_{\k}  
\widehat{T}(H_1(M,\k))$ the degree $0$ part of $\omega$.

Now let $G=\pi_1(M,x_0)$ and suppose $\widehat{\gr}(\k{G})$ 
admits a presentation of the form $\widehat{T}(H_1(M,\k))/ I$, 
for some closed Hopf ideal $I$ in the completed tensor algebra on $H_1(M,\k)$.  
If the connection $\omega_0$ is flat modulo the relations in $I$, 
the corresponding holonomy  homomorphism,  
$J\colon G\to \widehat{T}(H_1(M,\k))/I$, 
may be defined by means of iterated integrals, as follows:
\begin{align}
\label{eq:chenint}
J(g)&=
1+\sum_{k=1}^{\infty} \int_{0\leq t_1\leq \cdots\leq t_k\leq 1}
\omega_0(\dot{\gamma}(t_k))\wedge \cdots\wedge \omega_0(\dot{\gamma}(t_1))\\ \notag
&=1+\sum_{k=1}^{\infty} \int_{0}^{1}
\omega_0(\dot{\gamma}(t_k))    \cdots \int_{0}^{t_3}\omega_0(\dot{\gamma}(t_2))
\int_{0}^{t_2}\omega_0(\dot{\gamma}(t_1))\, ,
\end{align}
where $g\in G$ is represented by a piecewise smooth loop $\gamma\colon [0,1] \to M$ 
at $x_0$.  As shown in \cite{Chen57} (see also  \cite{Linxiaosong97,Kohno16}), 
the holonomy homomorphism $J\colon G\to  \widehat{\gr}(\k{G})$ 
is multiplicative and maps $G$ to group-like elements in
$\widehat{\gr}(\k{G})$; thus, $J$ is a Taylor expansion for $G$.

\subsection{Expansions of free groups}
\label{sec:magnus}

Let $F_n$ be a finitely-generated free group on generators $x_1,\dots,x_n$.
The complete Hopf algebra $\widehat{\gr}(\k F_n)$ can be identified with 
$\k\ncom{\mathbf{X}}=\k\ncom{X_1,\dots,X_n}$, 
the power series ring over $\k$ in $n$ non-commuting variables. 
There are three well-known expansions of this group.

\begin{enumerate}
\item The first one is the Magnus expansion, 
$M\colon F_n\to \k\ncom{\mathbf{X}}$, 
given by $M(x_i)=1+X_i$, see \cite{Magnus-K-S}.
This expansion is multiplicative but not co-multiplicative 
if $n>1$; thus, it is not a Taylor expansion.

\item The second one is the power series expansion, 
$L\colon F_n\to \k\ncom{\mathbf{X}}$, given by 
$L(x_i)= \exp(X_i)$.  As shown by Lin in  \cite{Linxiaosong97},  
this is a Taylor expansion.

\item The third type of expansion arises from the construction outlined in
\S \ref{subsec:chen}, with $\k=\C$. Let $C_n=\C\setminus \{1,\ldots,n\}$ 
be the complex plane $\C$ with $n$ punctures, so that $F_n=\pi_1(C_n, 0)$. Let 
$w_i=\dfrac{1}{2\pi \sqrt{-1}}\cdot\dfrac{dz}{z-i}$
be closed $1$-forms on $C_n$ dual to the cycles $x_i$. 
Then $\omega=\sum_{i=1}^n w_i X_i$ is a degree~$0$ flat connection on the 
trivial bundle $C_n\times \C\ncom{\mathbf{X}}\to C_n$.
The corresponding monodromy representation, 
$J\colon F_n\to \C\ncom{\mathbf{X}}$, 
is given by
\begin{equation}
\label{eq:jeff}
J(f)=1+ \sum_{k=1}^{\infty}\sum\limits_{0\leq i_1, \dots, i_k\leq n} 
\left(\dfrac{1}{2\pi \sqrt{-1}} \right)^k 
\int\limits_{0\leq t_1\leq \cdots \leq t_k\leq 1} \bigwedge_{r=1}^k 
\dfrac{d\gamma(t_r)}{\gamma(t_r)-i_r}\, X_{i_1}\cdots X_{i_k}\, ,
\end{equation}
where $f\in F_n$ is represented by a piecewise smooth loop $\gamma\colon [0,1] \to C_n$ at $0$.
This gives another Taylor expansion over $\C$ for the free group $F_n$.
\end{enumerate}

\subsection{Expansions of free abelian groups}
\label{sec:magnus-abel}
Let $\Z^n$ be the free abelian group of rank $n>0$.  
This group admits a presentation of the form 
$\Z^n=F_n/N$, where $N$ is the normal subgroup of $F_n$ generated
by the commutators $[x_i,x_j]:=x_ix_jx_i^{-1}x_j^{-1}$ for $1\leq i<j\leq n$.

The complete Hopf algebra $\widehat{\gr}(\k \Z^n)$ may be identified with 
$\k\com{\mathbf{X}}=\k\com{X_1,\dots,X_n}$, 
the power series ring over $\k$ in $n$ commuting variables. 
The power series expansion of the free group $F_n$ induces a Taylor expansion 
of the free abelian group $\Z^n$; this expansion, $L\colon \Z^n\to \k\com{\mathbf{X}}$, 
is given by $L(x_i)= \exp(X_i)$. 

\subsection{Taylor expansions for surface groups }
\label{subsec:Surfacegroups}

Let $G=\pi_{1}(S_g)=F_{2g}/\langle r\rangle$ be the fundamental 
group of a compact, connected, orientable surface of genus $g\ge 1$.  
Such a group has a presentation with generators
$x_i, y_i$ for $i=1, \dots, g$ and a single relator
$r=\sum_{i=1}^g[x_i,y_i]$.   It is well-known that $G$ is $1$-formal. 
In particular, there is a Taylor expansion $G\to \widehat{\gr}(\k{G})$, 
for any field $\k$ of characteristic $0$. 
Here, the complete Hopf algebra $\widehat{\gr}(\k G)$ is generated by
$X_i$, $Y_i$ for $i=1, \dots, g$, and subjects to a relation
$\sum_{i=1}^g[X_i,Y_i]=0$. 
However, actually constructing such an expansion is not an easy task. 

Using Chen's theory of iterated integrals, Lin constructed in \cite{Linxiaosong97} 
an explicit Taylor expansion over $\k=\C$ for the group $G=\pi_{1}(S_g)$. 
Let $\alpha_i, \beta_i$ be closed $1$-forms dual to $x_i, y_i$, respectively.
Set $\omega=\sum_{r=1}^{\infty}\omega^{(r)}$, where 
$\omega^{(1)}=\sum_{i=1}^g \alpha_iX_i+ \sum_{i=1}^g \beta_iY_i$,
and $\omega^{(r)}$ is the homogeneous polynomial of degree $r$ 
defined inductively by solving the equation $d \omega -\omega\wedge \omega=0$.
Then $\omega$ is a flat formal power series connection
on $S_g$.  The corresponding expansion, $J\colon G\to \widehat{\gr}(\k G)$, 
is defined by means of the iterated integral
\eqref{eq:chenint}.
By Theorem \ref{thm:Taylorpropagation}\eqref{item:fieldextensionT},
there exists a rational Taylor expansion for $G$.

Recently, Massuyeau \cite{Massuyeau12} constructed 
rational Taylor expansions for the surface groups  $G=\pi_{1}(S_g)$ 
by suitably deforming the power series expansion of the free groups $F_{2g}$. 
  
\section{Lower central series and holonomy Lie algebras}
\label{sec:malcev}

\subsection{Associated graded Lie algebras}
\label{subsec:Associatedgraded}
Let $G$ be a group. 
The \emph{lower central series}\/ of $G$ is the sequence 
of subgroups $\{\Gamma_k G\}_{k\geq 1}$ \/  defined inductively 
by $\Gamma_1G=G$ and 
\begin{equation}
\label{eq:lcs}
\Gamma_{k+1} G=[\Gamma_k G,G]
\end{equation}
for $k\geq 1$. 
Here, for any subgroups $H$ and $K$ of $G$, we denote $[H,K]$ 
the subgroup of $G$ generated by all group commutators
$[h,g]:=hgh^{-1}g^{-1}$ with $h\in H$ and $g\in K$. In particular, 
$\Gamma_{2}G$ equals $G'$, the commutator subgroup of $G$. 
Clearly, each term in the LCS series is a normal subgroup 
(in fact, a characteristic subgroup) of $G$.  Moreover, 
$\Gamma_{k+1} G$ contains the commutator subgroup of $\Gamma_k G$, 
and so the quotient group, $\gr_k(G):=\Gamma_k G/\Gamma_{k+1} G$,  
is abelian. 

Let us fix a coefficient field $\k$ of characteristic $0$.  The 
associated graded Lie algebra of $G$ over $\k$ is defined by
\begin{equation}
\label{eq:gradedLiealgebra}
\gr(G,\k):=\bigoplus\limits_{k\geq 1} \gr_k(G) \otimes \k,
\end{equation}
with the Lie bracket induced by the group commutator. 
This construction is functorial: if $\varphi\colon G\to H$ 
is a group homomorphism, then $\varphi$ preserves the respective 
lower central series, and so it induces a morphism of graded Lie algebras, 
 $\gr(\varphi, \k)\colon \gr(G,\k) \to \gr(H,\k)$.

Assume now that $G$ is a finitely generated group.  Then each 
LCS quotient $\gr_k(G)$ is a finitely generated abelian group.   
Furthermore, $\gr(G,\k)$ is a finitely generated graded Lie algebra, 
that can be presented as $\gr(G,\k)=\Lie(V)/\fr$, where $\Lie(V)$ 
is the free Lie algebra on a finite-dimensional $\k$-vector space $V$ 
(with non-zero elements  in degree $1$), and $\fr$ is a homogeneous 
Lie ideal.  We let $\phi_k(G):=\dim \gr_k(G,\k)$ be the LCS ranks of $G$.

\subsection{Chen Lie algebras}
\label{subsec:ChenLie}
Another descending series associated to a group $G$ is the 
{\em derived series}, starting at 
$G^{(0)}=G$,  $G^{(1)}=G'$,  and $G^{(2)}=G''$, 
and defined inductively by $G^{(i+1)}=[G^{(i)},G^{(i)}]$.  Note 
that any homomorphism $G\to H$ takes $G^{(i)}$ to $H^{(i)}$. 
The quotient groups, $G/G^{(i)}$, are solvable; in particular, $G/G'=G_{\ab}$, 
while $G/G''$ is the maximal metabelian quotient of $G$. 

Assume now that $G$ is finitely generated. 
For each $i\ge 2$, the \textit{$i$-th Chen Lie algebra}\/ of $G$ is defined to 
be the associated graded Lie algebra of the corresponding solvable quotient, 
\begin{equation}
\label{eq:Chen Lie}
\gr(G/G^{(i)},\k).
\end{equation} 
Clearly, this construction is functorial. 
The quotient map, $p_i\colon G\surj G/G^{(i)}$, induces a surjective 
morphism $\gr(p_i)$ between associated graded Lie algebras
$\gr_k(G,\k)$ and  $\gr_k(G/G^{(i)},\k)$.  Plainly,  
this morphism is the canonical identification in degree $1$.  
In fact, the map  $\gr(p_i)$ is an isomorphism for each 
$k\leq 2^i-1$, see \cite{SW-holo}.

We now specialize to the case when $i=2$, 
originally studied by K.-T. Chen in \cite{Chen51}. 
The {\em Chen ranks}\/ of $G$ are defined as 
$\theta_k(G):=\dim_{\k} (\gr_k(G/G^{''},\k))$. 
By the above remarks, $\phi_k(G)\ge \theta_k(G)$, with 
equality for $k\le 3$. 

\subsection{Holonomy Lie algebras}
\label{subsec:holoLie}
Once again, let $G$ be a finitely generated group. Write $V=H_1(G,\k)$ 
and let $\mu_G^{\vee}\colon H_2(G,\k)\to V\wedge V$ be the dual of 
the cup product map $\mu_G\colon H^1(G,\k)\wedge H^1(G,\k) \to H^2(G,\k)$.
The \emph{holonomy Lie algebra}\/ of $G$ is the 
quadratic Lie algebra defined as
\begin{equation}
\label{eq:HolonomyLieAlgebra}
\fh(G,\k)=\lie(V)/\langle \im \mu_G^{\vee} \rangle\, .
\end{equation}
Clearly, this construction is functorial. Furthermore, 
there is a natural surjective morphism of graded Lie algebras, 
\begin{equation}
\label{eq:ComparisonMap}
\xymatrixcolsep{16pt}
\xymatrix{
\psi_G\colon \fh(G,\k) \ar@{->>}[r]& \gr(G,\k)\, ,
}
\end{equation}
inducing isomorphisms in degree $1$ and $2$. (See \cite[Lemma 6.1]{SW-holo} 
and references therein.) If the map $\psi_G$ is an isomorphism, then we say that 
the group $G$ is \emph{graded-formal} (over $\k$). 

\subsection{Free groups and surface groups}
\label{subsec:free surf}

We conclude this section with some simple examples illustrating 
the above concepts. 

\begin{example}
\label{ex:gr-free}
Let $F_n=\langle x_1,\dots, x_n\rangle$ be the free group 
of rank $n$.  Then $\gr(F_n,\k)=\Lie(x_1,\dots,x_n)$, the free 
Lie algebra on $n$ generators, and the map $\psi\colon 
\h(F_n, \k)\to \gr(F_n,\k)$ is an isomorphism.  Moreover, as 
shown by Witt \cite{Witt37} and Magnus \cite{Magnus37},  the LCS ranks 
are given by  
\begin{equation}
\label{eq:lcs-free}
\prod\nolimits_{k\geq 1}(1-t^k)^{\phi_k(F_n)}=1-n t ,
\end{equation}
or, equivalently, $\phi_k(F_n)=\tfrac{1}{k}\sum_{d\mid k} \mu(d) n^{k/d}$, 
where $\mu$ denotes the M\"{o}bius function.  Finally, as shown in 
\cite{Chen51}, the Chen ranks of the free groups are given by $\theta_1(F_n)=n$ 
and 
\begin{equation}
\label{eq:chen-free}
\theta_k(F_n)=(k-1)\binom{n+k-2}{k}
\quad
\text{for $k\ge 2$}.
\end{equation}
\end{example}

\begin{example}
\label{ex:gr-surface}
Let $S_g$ be a closed, orientable surface of genus $g\ge 1$. 
Its fundamental group, $\Pi_g=\pi_1(S_g)$, has a presentation 
with generators $ x_1,y_1,\dots, x_g,y_g$ and a single relator, 
$[x_1,y_1]\cdots [x_g,y_g]$.  As shown by Labute \cite{Labute70}, 
$\gr(\Pi_g,\k)= \Lie(x_1,y_1,\dots,x_g,y_g )/\langle 
\sum_{i=1}^{g} [x_i,y_i]\rangle$. Again, it is readily seen that 
$\h(\Pi_g), \k)\cong \gr(\Pi_g,\k)$. Furthermore, the LCS ranks 
of $\Pi_g$ are given by 
\begin{equation}
\label{eq:lcs-surf}
\prod\nolimits_{k\geq 1}(1-t^k)^{\phi_k(\Pi_g)}=1-2gt+t^2, 
\end{equation}
while the Chen ranks are given by $\theta_1(\Pi_g)=2g$, 
$\theta_2(\Pi_g)=2g^2-g-1$, and 
\begin{equation}
\label{eq:chen-surf}
\theta_k(\Pi_g)=(k-1)\binom{2g+k-2}{k}-\binom{2g+k-3}{k-2}
\quad\text{for $k\ge 3$}.
\end{equation}
\end{example}

\section{Malcev Lie algebras and formality properties}
\label{sect:malcev-formal}

\subsection{Malcev Lie algebras}
\label{subsec:malcev}
As before, let  $G$ be a finitely generated group, let $\k$ be a field 
of characteristic $0$, and let $\widehat{\k{G}}$ 
be the $I$-adic completion of the group algebra of $G$, where $I$ is the 
augmentation ideal of $\k{G}$.  
Following Quillen \cite{Quillen69}, we define the \emph{Malcev Lie algebra}\/ 
of $G$ as the set $\fm(G,\k)$ of all primitive elements in $\widehat{\k{G}}$, 
with bracket $[x,y]=xy-yx$.  By construction, $\fm(G,\k)$ is a complete, filtered 
Lie algebra.  Moreover, if we complete 
 the universal enveloping algebra $U(\fm(G,\k))$ 
with respect to the powers of its augmentation ideal, then 
$\widehat{U}(\fm(G,\k))\cong \widehat{\k{G}}$, as
complete Hopf algebras.

The set of all primitive elements in $\gr(\k{G})$ forms a graded Lie algebra,
which is isomorphic to $\gr(G,\k)$.
An important connection between the Malcev Lie algebra $\fm(G,\k)$ 
and the associated graded Lie algebra $\gr(G;\k)$ was discovered 
by Quillen, who showed in \cite{Quillen68} that there is an isomorphism 
of graded Lie algebras, 
\begin{equation}
\gr(\fm(G,\k))\cong \gr(G,\k).
\end{equation}

The set of all group-like elements in 
$\widehat{\k{G}}$ forms a group, denoted $\fM(G;\k)$.  This group comes 
endowed with a complete, separated filtration, whose $k$-th term 
is $\fM(G;\k) \cap (1+\widehat{I^k})$.  
As explained for instance in \cite{Massuyeau12}, 
there is a one-to-one, filtration-preserving 
correspondence between primitive elements 
and group-like elements of $\widehat{\k{G}}$ via the exponential 
and logarithmic maps 
\begin{equation}
\label{eq:explog}
\xymatrix{
 \fM(G;\k)\subset 1+\widehat{I} \ar@/_0.9pc/[rr]|{~\log~} &  
  & \widehat{I} \supset\fm(G;\k)\ar@/_1.2pc/[ll]|{~\exp~} }.
\end{equation} 

Let $G$ be a group which admits a finite presentation of the form $G=F/R$. 
Using a Taylor expansion for the finitely generated free group $F$, we may find 
a presentation for the Malcev Lie algebra 
$\fm(G;\k)$, using the approach of Papadima \cite{Papadima95} and 
Massuyeau \cite{Massuyeau12}, which may be summarized in 
the following theorem.

\begin{theorem}[\cite{Massuyeau12, Papadima95}]
\label{thm:Massuyeau}
Let $G$ be a group with generators $x_1,\dots,x_n$ and relators $r_1,\dots, r_m$.
Let $E$ be a Taylor expansion of the free group $F=\langle x_1,\dots,x_n\rangle$.
There exists then a unique filtered Lie algebra isomorphism
\[
\fm(G;\k) \cong\widehat{\Lie}(\k^n)/
\langle\!\langle W\rangle\!\rangle,
\]
where $\langle\!\langle W \rangle\!\rangle$ denotes 
the closed ideal of the completed free Lie algebra $\widehat{\Lie}(\k^n)$ generated 
by the subset $\{\log(E(r_1)),\dots,\log(E(r_m))\}$. 
\end{theorem}

\subsection{Formality and filtered-formality}
\label{subsec:ff}

The notion of formality first appeared in the study of rational 
homotopy types of topological spaces initiated by Sullivan \cite{Sullivan, DGMS}. 
Since then, it has been broadly used in investigating a variety 
of differential graded objects.   We recall now a formality notion 
introduced in \cite{SW-formal}.

\begin{definition}
\label{def:MalcevLie}
A finitely generated group $G$ is called \emph{filtered-formal} (over $\k$), 
if there is a filtered Lie algebra isomorphism from the Malcev Lie algebra $\fm(G;\k)$ 
to the degree completion $\widehat{\gr}(G;\k)$ inducing the identity on associated 
graded Lie algebras.
\end{definition}
 
As shown in \cite[Lemma 2.4]{SW-formal}, the following holds: 
if $\fm(G;\k)$ is isomorphic (as a filtered Lie algebras) to the 
degree completion of a graded Lie algebra $\fg$,
then the group $G$ is  filtered-formal  (over $\k$). 
The notion of filtered-formality satisfies the following 
propagation properties. 

\begin{theorem}[\cite{SW-formal}]
\label{thm:propagation}
Let $G$ be a finitely generated group. 
\begin{enumerate}
\item \label{item:splitinjection} Suppose there is a split 
monomorphism $\iota\colon K\rightarrow G$.  
If $G$ is filtered-formal, then $K$ is also filtered-formal. 
\item \label{item:fieldextension}   The group $G$ is filtered-formal over 
a field $\k$ of characteristic $0$ if and only if $G$ is filtered-formal over $\Q$.
\item \label{item:products}  
$G_1$ and $G_2$ are filtered-formal if and only if
 $G_1* G_2$  is  filtered-formal if and only if
 $G_1\times G_2$ is filtered-formal. 
\end{enumerate} 
\end{theorem}
\begin{proof}
This theorem is a combination of the following results 
from \cite{SW-formal}:  
Theorem 5.11 for \eqref{item:splitinjection};
Theorem 6.6 for \eqref{item:fieldextension}; 
Theorem 7.17 for \eqref{item:products}.
\end{proof} 

In particular, if a finitely generated group $G$ if filtered-formal
over $\C$, then it also filtered-formal over $\Q$.

A finitely generated group  group $G$ is said to be \emph{$1$-formal}\/ (over $\k$)
if $\fm(G,\k)\cong\widehat{\fh}(G,\k)$ as filtered Lie algebras. It is readily seen 
that $G$ is $1$-formal if and only if it is 
graded-formal and filtered-formal.

\subsection{Chen Lie algebras and formality}
\label{subsec:chen formal}

The next theorem is the Lie algebra version of Theorem 3.5 from \cite{PS04-imrn}, 
which describes the relationship between the Malcev Lie algebras of the derived quotients 
of a group $G$ and the corresponding quotients of the Malcev Lie algebra of $G$. 

\begin{theorem}[\cite{PS04-imrn}]
\label{thm:DerivedQuotientMalcev}
Let $G$ be a finitely generated group. 
There is an isomorphism of complete, separated filtered Lie algebras,
\begin{equation*}
\fm(G/G^{(i)};\k)\cong \fm(G;\k)/\overline{\fm(G;\k)^{(i)}},
\end{equation*}
for each $i\ge 2$, where $\overline{\fm(G;\k)^{(i)}}$ is the closure of $\fm(G;\k)^{(i)}$ 
with respect to the filtration topology on $\fm(G;\k)$.
\end{theorem}

One important application of Theorem \ref{thm:DerivedQuotientMalcev} is the 
next theorem, which delineates the relationship between associated graded Lie algebras of 
derived quotients and derived quotients of associated graded Lie algebras. 
This theorem also shows that filtered-formality is preserved under the operation 
of taking derived quotients. 

\begin{theorem}[\cite{SW-formal}]
\label{thm:DerivedQuotientMalcevISO}
The quotient map $p_i\colon G\surj G/G^{(i)}$ induces a 
natural epimorphism of graded $\k$-Lie algebras, 
\begin{equation*}
\xymatrix{\Psi_G^{(i)}\colon \gr(G;\k)/\gr(G;\k)^{(i)} \ar@{->>}[r]
& \gr(G/G^{(i)};\k)},
\end{equation*}
for each $i\ge 2$.
Moreover, if the group $G$ is filtered-formal, then $\Psi_G^{(i)}$ is 
an isomorphism and the derived quotient $G/G^{(i)}$ is filtered-formal.
\end{theorem}

\subsection{filtered-formality and Chen Lie algebras}
\label{subsec:ff-chen}
As mentioned previously, any homomorphism $G_1\to G_2$ induces morphisms 
of graded Lie algebras, 
$\gr(G_1;\k)\to \gr(G_2;\k)$ and $\gr(G_1/G_1^{(i)};\k)\to \gr(G_2^{(i)};\k)$.
On the other hand, it is not {\it a priori}\/ clear that a morphism 
$\gr(G_1;\k)\to \gr(G_2;\k)$ should induce 
morphisms between the corresponding Chen Lie algebras. 
Nevertheless, as the next theorem shows, this happens 
for filtered-formal groups.

\begin{theorem}
\label{thm:Chenobstruction}
Let $G_1$ and $G_2$ be two $\k$-filtered-formal groups. 
Then every morphism of graded Lie algebras,  
$\alpha \colon \gr(G_1;\k)\to \gr(G_2,\k)$, induces 
morphisms $\alpha_i\colon \gr(G_1/G_1^{(i)};\k)\to \gr(G_2/G_2^{(i)};\k)$  
for all $i\ge 1$. Consequently, if  $\gr(G_1;\k)\cong \gr(G_2;\k)$, then
$\gr(G_1/G_1^{(i)};\k)\cong \gr(G_2/G_2^{(i)};\k)$, for all $i$.
\end{theorem} 

\begin{proof}
Fix an index $i\ge 1$, and consider the following diagram of 
graded Lie algebras:  
\begin{equation}
\begin{gathered}
\xymatrix{
\gr(G_1;\k)\ar[d]^{\alpha} \ar@{->>}[r] &\gr(G_1;\k)/\gr(G_1;\k)^{(i)} 
\ar[r]^(.58){\Psi_{G_1}^{(i)}}\ar@{.>}[d]^{\beta_i} &\gr(G_1/G_1^{(i)};\k)\ar@{.>}[d]^{\alpha_i}\\ 
\gr(G_2;\k) \ar@{->>}[r]  &\gr(G_2;\k)/\gr(G_2;\k)^{(i)} 
\ar[r]^(.58){\Psi_{G_2}^{(i)}} &\gr(G_2/G_2^{(i)};\k)
}
\end{gathered}
\end{equation}
The morphism $\alpha$ induces a morphism $\beta_i$ between the 
respective solvable quotients.  By Theorem \ref{thm:DerivedQuotientMalcevISO}, 
the maps $\Psi_{G_1}^{(i)}$ and $\Psi_{G_2}^{(i)}$
are isomorphisms. We define the desired morphism $\alpha_i$
to be the composition $\Psi_{G_2}^{(i)}\circ\beta_i\circ \big(\Psi_{G_1}^{(i)}\big)^{-1}$. 
The last claim follows at once.
\end{proof}

Taking $i=2$ in the above theorem, we obtain the following corollary.

\begin{corollary}
\label{cor:noniosLie}
Suppose $G_1$ and $G_2$ are two $\k$-filtered-formal groups. 
If $\theta_k(G_1)\neq \theta_k(G_2)$ for some $k\geq 1$, 
then $\gr(G_1,\k)\not\cong \gr(G_2,\k)$, as graded Lie algebras.
\end{corollary}

\section{Taylor expansions and formality properties}
\label{sec:expformal}

In this section we relate the notions of Taylor expansion and filtered-formality 
for a finitely generated group $G$.

\subsection{Taylor expansions and isomorphisms of filtered Lie algebras}
\label{subsec:tff}

As the next theorem shows, Taylor expansions of $G$ 
are intimately related to isomorphisms between the Malcev Lie 
algebra $\fm(G;\k)$ and the LCS completion of the associated 
graded Lie algebra $\gr(G;\k)$.

\begin{theorem}
\label{thm:expansionFiltered}
There is a one-to-one correspondence between Taylor expansions 
$G\to \widehat{\gr}(\k G)$ and filtration-preserving Lie algebra isomorphisms 
$\fm(G;\k)\to \widehat{\gr}(G;\k)$ inducing the identity on $\gr(G,\k)$.  
\end{theorem}

\begin{proof}
First suppose $E\colon G\to \widehat{\gr}(\k G)$ is a Taylor expansion.  
Then, by Proposition \ref{prop:Taylor}, 
there is a filtration-preserving Hopf algebra isomorphism   
$\widehat{E}\colon \widehat{\k{G}} \to \widehat{\gr}(\k G)$, inducing the identity 
on $\gr(\k{G})$.   
Recall that $\widehat{\k{G}}\cong U(\fm(G;\k))$ and $\widehat{\gr}(\k G)\cong
U(\widehat{\gr}(G;\k))$, as filtered Hopf algebras.   Taking primitives, 
we obtain a filtration-preserving  isomorphism of complete Lie algebras, 
$\Prim(\widehat{E})\colon \fm(G;\k)\to \widehat{\gr}(G;\k)$,  inducing the identity 
on $\gr(G;\k)$.   
 
Now suppose there is an isomorphism of filtered, complete Lie algebras, 
$\alpha\colon \fm(G;\k)\to \widehat{\gr}(G;\k)$, such that $\gr(\alpha)=\id$.
Taking universal enveloping algebras, we obtain an 
isomorphism of filtered, complete Hopf algebras,  
$U(\alpha)\colon\widehat{\k G}\isom\widehat{\gr}(\k G)$, 
such that $\gr(\phi)=\id$. 
By Proposition \ref{prop:TaylorExp}, the map $U(\alpha)$ induces  
a Taylor expansion $E\colon G\to \widehat{\gr}(\k G)$.
\end{proof}

Using this theorem, we obtain in Corollaries \ref{cor:TaylorFilteredFormal} 
and \ref{cor:expansionFormal} alternate interpretations of filtered-formality and
$1$-formality.

\begin{corollary}
\label{cor:TaylorFilteredFormal}
A finitely generated group $G$ has a Taylor expansion if and only if $G$ is filtered-formal. 
\end{corollary}

\begin{proof}
 Follows at once from Theorem \ref{thm:expansionFiltered} and Definition \ref{def:MalcevLie}. 
\end{proof}

\begin{theorem} 
\label{thm:Taylorpropagation}
Let $G$ be a finitely generated group. 
\begin{enumerate}
\item \label{item:splitinjectionT} 
Suppose there is a split monomorphism $\iota\colon K\rightarrow G$.  
If $G$ has a Taylor expansion, then $K$ also has a Taylor expansion. 

\item \label{item:fieldextensionT}   The group $G$ has a Taylor expansion 
over a field $\k$ of characteristic $0$ if and only if $G$ has a Taylor expansion 
over $\Q$.

\item \label{item:productsT}  
$G_1$ and $G_2$ have a Taylor expansion if and only if
 $G_1* G_2$  has a Taylor expansion if and only if
 $G_1\times G_2$ has a Taylor expansion. 
 
 \item \label{item:solvablequo}  If $G$ has a Taylor expansion, then 
all the solvable quotients $G/G^{(i)}$ have a Taylor expansion.
\end{enumerate} 
\end{theorem}
\begin{proof}
The first three claims follow from Corollary \ref{cor:TaylorFilteredFormal}
and Theorem \ref{thm:propagation}. 
Claim \eqref{item:solvablequo} follows from Corollary \ref{cor:TaylorFilteredFormal}
and Theorem \ref{thm:DerivedQuotientMalcevISO}. 
\end{proof}
 
\begin{corollary}
\label{cor:expansionFormal}
A finitely generated group $G$ is $1$-formal if and only if there is  
a Taylor expansion $G \to  \widehat{\gr}(\k{G})$ 
and $\gr(\k G)$ is a quadratic algebra.
\end{corollary}

\begin{proof}
We know that $G$ is $1$-formal if 
and only if $G$ is filtered-formal and graded-formal.  By 
Corollary \ref{cor:TaylorFilteredFormal}, 
$G$ is filtered-formal if and only if it has a Taylor expansion.
On the other hand, $G$ is graded-formal if and only if $\gr(G;\k)$ admits a
quadratic presentation.  
As shown in \cite[\S 2.2.3]{Lee}, this latter condition is equivalent to
the quadraticity of $\gr(\k{G})$. This completes the proof. 
\end{proof}

\begin{example}
\label{ex:ReducedFree}
The reduced free group $RF_n$, introduced by J.~Milnor 
in his study of link homotopy \cite{Milnor54} is the quotient of the free group 
 $F_n=\langle x_1,\dots, x_n\rangle$ by the  normal subgroup 
generated by all elements of the form $[x_i, gx_ig^{-1}]$ with $g\in F_n$.
The relations in $RF_n$ can be reduced to multiple group commutators 
in $x_1, \dots, x_n$ with some $x_i$ appears at least twice.
In \cite{Linxiaosong97}, Lin showed that $RF_n$ has 
Taylor expansions induced from certain expansions of the 
free group $F_n$ (the power series expansion and 
the expansion arising from formal power series connections, 
as described in \S\ref{sec:magnus}).
It follows from Corollary \ref{cor:TaylorFilteredFormal} that the 
group $RF_n$ is filtered-formal.
\end{example}

\subsection{Taylor expansions of nilpotent groups }
\label{sec: TaylorNilpotent}

As before, let $G$ be a finitely generated group. 
The next result shows that the Taylor expansions of $G$ are 
inherited by the nilpotent quotients $G/\Gamma_{i}G$.

\begin{theorem}
\label{thm:nilp-taylor}
Suppose $G$ admits a 
Taylor expansion $E\colon G\to \widehat{\gr}(\k G)$.  
Then each nilpotent quotient $G/\Gamma_{i}G$ admits 
an induced Taylor expansion,  
$E_i\colon G/\Gamma_{i}G\to \widehat{\gr}(\k [G/\Gamma_iG])$.
\end{theorem}

\begin{proof}
By Theorem \ref{thm:expansionFiltered}, the Taylor expansion 
$E\colon G\to \widehat{\gr}(\k G)$ determines a filtered 
Lie algebra isomorphism, $\alpha\colon \fm(G;\k)\to \widehat{\gr}(G;\k)$. 
From the proof of \cite[Theorem 7.13]{SW-formal}, we deduce that 
$\alpha$ induces filtered Lie algebra isomorphisms,
$\alpha_i\colon \fm(G/\Gamma_iG;\k)\to \widehat{\gr}(G/\Gamma_iG;\k)$.
Using Theorem \ref{thm:expansionFiltered} again, we obtain the desired 
Taylor expansions, $E_i\colon G/\Gamma_{i}G \to \widehat{\gr}(G/\Gamma_{i}G;\k)$.
\end{proof}

\begin{example}
\label{ex:nstep}
As noted in \S \ref{sec:magnus}, the finitely generated free group $F$ admits 
Taylor expansions. By Theorem \ref{thm:nilp-taylor}, 
the $k$-step, free nilpotent group $F/\Gamma_{k+1}F$ admits Taylor expansions
for each $k\geq 1$. 
\end{example}
 
\begin{example}
\label{ex:step2nilpotent}
Let $G$ be a finitely generated, torsion-free, $2$-step nilpotent group, 
and suppose $G_{\ab}$ is also torsion-free.  As shown in \cite{SW-formal}, 
the group $G$ is filtered-formal.  Thus, by Corollary \ref{cor:TaylorFilteredFormal}, 
$G$ admits a Taylor expansion. 
\end{example}

\begin{example}
\label{ex:Cornulier}
Let $\fm$ be the $5$-dimensional, nilpotent Lie algebra with non-zero 
Lie brackets given by $[e_1,e_3]=e_4$ and $[e_1,e_4]=[e_2,e_3]=e_5$. 
This Lie algebra may be realized as the Malcev Lie algebra of 
a finitely generated, torsion-free nilpotent group $G$.  
As noted in \cite{Cornulier14, SW-formal}, this group is not filtered-formal. 
Thus, the group $G$ admits no Taylor expansion.
\end{example}

\section{Automorphisms of free groups and almost-direct products}
\label{sec:exp-rtfn}

\subsection{Braid groups}
\label{subsec:BraidGroups}

An automorphism of the free group $F_n=\langle x_1,\dots, x_n\rangle$ 
is a permutation-conju\-gacy automorphism if it sends each generator 
$x_i$ to a conjugate of some other generator $x_j$. 
The Artin braid group $B_n$ is the subgroup of $\Aut(F_n)$ 
consisting of all permutation-conjugacy automorphisms 
which fix the product $x_1\cdots x_n$.   As shown for instance in \cite{Birman74},   
the group $B_n$ is generated by the elementary braids 
$\sigma_1,\dots ,\sigma_{n-1}$ (where $\sigma_i$ sends $x_i$ to $x_{i+1}$ 
and $x_{i+1}$ to $x_{i+1}^{-1}x_ix_{i+1}$ while fixing the other $x_k$'s), 
subject to the relations 
\begin{equation}
\label{eq:braidrelations}
\begin{cases}
\sigma_{i} \sigma_{j}= 
\sigma_{j} \sigma_{i}, &\abs{i-j}\geq 2,\\
\sigma_{i} \sigma_{i+1} \sigma_{i} = 
\sigma_{i+1} \sigma_{i} \sigma_{i+1},
& 1\le i\le n-2.
\end{cases}
\end{equation}

The pure braid group $P_n$ is the kernel of the canonical projection  
$B_n\to S_n$ that sends a generator $\sigma_i$ to the 
transposition $(i,i+1)$.  This group is generated by the braids 
\begin{equation}
\label{eq:GeneratorsPureBraid}
A_{ij}:=(\sigma_{j-1}\sigma_{j-2}\cdots\sigma_{i+1}) \sigma_i^2 
(\sigma_{j-1}\sigma_{j-2}\cdots\sigma_{i+1})^{-1}, 
\textrm{ for } 1\leq i<j\leq n.
\end{equation}
The pure braid group $P_n$ decomposes as a semidirect product, 
$P_n=F_{n-1}\rtimes P_{n-1}$, where $P_{n-1}$ acts  on $F_{n-1}$ 
by restriction of the Artin representation $B_{n-1}\subset \Aut(F_{n-1})$.  
The group $P_n$ is $1$-formal.  The associated graded Lie algebra $\gr(P_n,\k)$ 
is generated by $\{t_{ij} \mid 1\leq i < j\leq n\}$ subject to the relations 
$[t_{ij}, t_{kl}]=0$ and $[t_{ij}, t_{ik}+t_{jk}]=0$ whenever $i,j,k,l$ are distinct.

\subsection{Taylor expansions for the pure braid groups}
\label{subsec:TaylorPn}

Explicit Taylor expansions for the pure braid groups $P_n$ over 
$\k=\C$ can be constructed using Chen's method of iterated integrals, 
see e.g. \cite{Linxiaosong97, Bar-Natan16, Kohno16}.  
Let $\textrm{Conf}_n(\C)=\C^n\setminus \bigcup_{1\leq i< j\leq n} \{ z_i= z_j \}$ 
be the configuration space of $n$ ordered points in $\C$, 
so that $P_n=\pi_1(\textrm{Conf}_n(\C), 0)$. Consider 
the logarithmic $1$-forms on $\textrm{Conf}_n(\C)$ given by
\begin{equation}
\label{eq:log-forms}
w_{ij}=w_{ji}=\dfrac{1}{2\pi \sqrt{-1}}\cdot d \log(z_i-z_j)\, .
\end{equation}
Clearly, these $1$-forms are closed.  Furthermore, as shown by Arnold 
\cite{Arnold69}, these $1$-forms satisfy the relations
$w_{ij}\wedge w_{jl}+w_{jl}\wedge w_{li}+w_{li}\wedge w_{ij}=0$. 

As shown by Kohno \cite{Kohno85}, 
the complete Hopf algebra $\widehat{\gr}(\k{P_n})$ admits a presentation 
with generators $\{X_{ij}=X_{ji}; 1\leq i<j\leq n\}$, 
subject to the infinitesimal pure braid relations 
\begin{equation}
\label{eq:purebraidalgebrarelations}
\begin{cases}
[X_{ij}, X_{kl}]=0 \\
[X_{ij}, X_{il}+X_{lj}]=0\, .
\end{cases}
\end{equation}

The formal power series connection $\omega=\sum_{1\leq i<j\leq n} w_{ij} X_{ij}$ 
on $\textrm{Conf}_n(\C)$ is flat.  The corresponding monodromy representation
yields a (faithful) Taylor expansion for the pure braid group, 
$J\colon P_n\to \widehat{\gr}(\k P_n)$, given by \eqref{eq:chenint}, more explicitly, 
as stated in \cite{Bar-Natan16}, (first appeared  in \cite{Kohno88} )
\begin{equation}
\label{eq:chenint-pn}
J(g)=1+ \sum_{k=1}^{\infty}\sum\limits_{1\leq i_1<j_1, \dots, i_k<j_k\leq n} \left(\dfrac{1}{2\pi \sqrt{-1}} \right)^k 
\int\limits_{0\leq t_1\leq \cdots \leq t_k\leq 1} \bigwedge_{s=1}^k 
 d \log(z_{i_s}-z_{j_s} )\, X_{i_1,j_1}\cdots X_{i_k,j_k}\, ,
\end{equation}
where $g\in P_n$ is represented by a piecewise smooth loop 
$\gamma\colon [0,1] \to \textrm{Conf}_n(\C)$ at $0$,
and $z_i$ is the $i$-th coordinate of the loop $\gamma$.

The Taylor expansion $J$ is called the monodromy of the flat connection in \cite{Kohno85}, 
and the holonomy homomorphism in \cite{Kohno16}. This expansion is a finite type invariant 
for the pure braid groups, and a prototype for the Kontsevich integral in knot theory.

\subsection{Welded braid groups}
\label{subsec:WeldedBraidGroups}

The welded braid group (or, the braid-permutation group) 
$\wB_n$ is the subgroup of $\Aut(F_n)$ 
consisting of all permutation-conjugacy automorphisms of $F_n$.  
The welded pure braid group (also known as the group of 
basis-conjugating automorphisms, or McCool group) $\wP_n$ 
is the kernel of the canonical projection $\wP_n\to S_n$. 
In \cite{McCool86}, McCool gave a finite presentation for  
$\wP_n$; the generators are the automorphisms $\alpha_{ij}$ 
($1\leq i\neq j\leq n$) sending $x_i$ to $x_jx_ix_j^{-1}$. 

The subgroup of $\wP_n$ generated by the elements $\alpha_{ij}$ with 
$i>j$ is called the {\em upper welded pure braid group} (or, upper 
triangular McCool group), and is denoted by $\wP_n^+$. 
As shown in \cite{Cohen-P-V-Wu08}, the upper welded 
pure braid group $\wP_n^+$ also decomposes as a 
semidirect product, $\wP_n^+=F_{n-1}\rtimes \wP_{n-1}^+$. 

Work of Berceanu and Papadima from \cite{Berceanu-Papadima09} 
establishes the $1$-formality of the groups $\wP_n$  and $\wP_n^+$.  
Bar-Natan and Dancso, in \cite{Bar-Natan14}, investigate expansions
of welded braid groups. The Chen ranks of the 
groups $P_n$, $\wP_n$,  and $\wP_n^+$ were computed 
in \cite{Cohen-Suciu95},  \cite{Cohen-Schenck15}, and  \cite{SW-mccool}, 
respectively.  We summarize those results, as follows. 

\begin{theorem}[\cite{Cohen-Suciu95,Cohen-Schenck15,SW-mccool}]
\label{thm:ChenRanksLCSRanks}
The Chen ranks of $P_n$, $\wP_n$, and $\wP_n^+$ are given by 
\begin{enumerate}
\item
$\theta_1(P_n)=\binom{n}{2}$, $\theta_2(P_n)=\binom{n}{3}$, and 
$\theta_k(P_n)= (k-1)\binom{n+1}{4}$ for $k\geq 3$.
\item
$\theta_k(\wP_n) = (k-1)\binom{n}{2} + (k^2-1) \binom{n}{3}$  for $k\gg 0$.
\item
$\theta_1(\wP_n^+)=\binom{n}{2}$,  $\theta_2(\wP_n^+)=\binom{n}{3}$, 
and $\theta_k(\wP_n^+)= \binom{n+1}{4} + \sum_{i=3}^k\binom{n+i-2}{i+1}$ 
for $k\geq 3$.
\end{enumerate} 
\end{theorem} 

\subsection{Distinguishing some related Lie algebras}
\label{subsec:discuss}

Both the pure braid groups $P_n$ and the upper McCool groups $\wP_n^+$ 
are iterated semidirect products of the form $F_{n-1}\rtimes \dots \rtimes F_2\rtimes F_1$. 
Clearly, $P_1=\wP_1^+ =\{1\}$ and $P_2=\wP_2^+ =\Z$; it is also known that 
$P_3\cong \wP_3^+\cong F_2\times F_1$.  Furthermore, 
both $P_n$ and $\wP_n^+$ share the same LCS ranks and the same Betti numbers
as the corresponding direct product of free groups, $\Pi_n= \prod_{i=1}^{n-1}F_{i}$, see
\cite{Arnold69, Cohen-P-V-Wu08, FalkRandell, Kohno85}.  

\begin{prop}
\label{prop:theta4}
For each $n\geq 4$, the graded Lie algebras $\gr(P_n,\k)$, $\gr(\wP^+_n,\k)$, 
and $\gr(\Pi_n,\k)$ are pairwise non-isomorphic. 
\end{prop}

\begin{proof}
Using the computations recorded in Theorem \ref{thm:ChenRanksLCSRanks}, 
we find that $\theta_4(P_n)=3\binom{n+1}{4}$ and 
$\theta_{4}(\wP_n^+)= 2\binom{n+1}{4}+\binom{n+2}{5}$. 
Furthermore, the computation of K.-T.~Chen recorded in Example \ref{ex:gr-free} 
implies that $\theta_4(\Pi_n)=3\binom{n+2}{5}$, cf. \cite{Cohen-Suciu95}. 

Comparing these ranks and using Corollary \ref{cor:noniosLie} 
shows that the graded Lie algebras $\gr(P_n,\k)$, $\gr(\Pi_n,\k)$, and 
$\gr(\wP^+_n,\k)$ are pairwise non-isomorphic, as claimed
\end{proof}

This proposition recovers (in stronger form) the following 
result from \cite{SW-mccool}: For each $n\geq 4$, 
the groups $P_n$,  $\wP_n^+$, and $\Pi_n$ are 
pairwise non-isomorphic.   

\subsection{Almost-direct products}
\label{subsec:Almostdirect}

A semi-direct product of groups, $H \rtimes Q$, is called
an \emph{almost-direct product} of $H$ and $Q$,
if the action of $Q$ on $H$ induces a trivial action on 
the abelianization $H_{\ab}$, that is,
$qhq^{-1} \equiv h$ modulo $[H,H]$
for any $q\in Q$ and $h\in H$.

\begin{theorem}
Let $G=H \rtimes Q$ be a almost-direct product. Then, 
\begin{enumerate}
\item[(1)]    $\gr(G;\k)\cong \gr(H;\k) \rtimes \gr(Q;\k)$  as graded Lie algebras. 

\item[(2)]   $\widehat{\gr}(\k G)\cong \widehat{\gr}(\k H)\hat{\otimes}
\widehat{\gr}(\k Q)$ as graded vector spaces.
\end{enumerate}
\end{theorem}

\begin{proof}
The first claim follows from \cite[Theorem (3.1)]{FalkRandell}, while the second claim 
follows from  \cite[Theorem 3.1]{Papadima02}. 
\end{proof}

In general, an almost-direct product of $1$-formal groups need not 
be $1$-formal, or even filtered-formal.

\begin{example}
\label{ex:semiproduct}
Let $L$ be the link of $5$ great circles in $S^3$ corresponding to the 
arrangement of transverse planes through the origin of $\R^4$ 
denoted as $\mathcal{A}(31425)$ in Matei--Suciu \cite{MateiSuciu00}.  
The link group $G=\pi_1(S^3\setminus L)$ is isomorphic to the almost-direct 
product $F_4\rtimes_{\alpha} F_1$, where $\alpha=A_{1,3}A_{2,3}A_{2,4}\in P_4$. 

From \cite{SW-holo}, based on the work of Berceanu and Papadima 
\cite{Berceanu-Papadima94}, the group $G$ is graded-formal.  On the 
other hand, as noted by Dimca, Papadima, and Suciu in 
\cite[Example 8.2]{Dimca-Papadima-Suciu}, the Tangent Cone theorem 
does not hold for this group, and thus $G$ is not $1$-formal.
Consequently, $G$ is not filtered-formal.
\end{example}

\section{Faithful Taylor expansions and the RTFN property}
\label{sect:RTFN}

\subsection{Residually torsion free-nilpotent groups}
\label{subsec:rtfn}

A group $G$ is said to be \emph{residually torsion-free nilpotent}\/ 
(for short RTFN) if for any $g\in G$, $g\neq 1$,
there exists a torsion-free nilpotent group $Q$, and an epimorphism
$\psi\colon G\to Q$ such that $\psi(g)\neq 1$.
Equivalently,  $G$ is residually torsion-free nilpotent if and only if
$\bigcap_{k\geq 1} \tau_{k}G=\{1\}$, where 
\begin{equation}
\label{eq:tau-filtration}
\tau_{k}G=\{g\in G\mid \text{$g^n \in \Gamma_{k}G$, for some $n\in \N$} \}.
\end{equation}
For a group $G$, the property of being residually torsion-free nilpotent 
is inherited by all subgroups, and is preserved under direct products and free products.

By \cite[Ch.~VI, Thm.~2.26]{Passi79}, a group $G$ is residually 
torsion-free nilpotent if and only if the group-algebra 
$\k{G}$ is residually nilpotent, that is, $\bigcap_{k\geq 1}I^k=\{0\}$, 
where $I$ is the augmentation ideal. 
Therefore, if $G$ is finitely generated, the RTFN condition is 
equivalent to the injectivity of the canonical map to the prounipotent 
completion, $\kappa\colon G\to \fM(G,\k)$, where
recall $\fM(G,\k)$ is the set of group-like elements in $\widehat{\k G}$.

If $G$ is residually nilpotent and $\gr_{k} (G)$ is torsion-free for 
$k\ge 1$, then $G$ is residually torsion-free nilpotent. 
Residually torsion-free nilpotent implies residually nilpotent, which
in turn implies residually finite. 
Examples of residually torsion-free nilpotent groups include 
torsion-free nilpotent groups, free groups and surface groups;  
more examples will be discussed below. 
 
\subsection{Torelli groups}
\label{subsec:torelli}
 
Let $G$ be a finitely generated group, and let $\Aut(G)$ be its 
group of automorphisms.  The {\em Torelli group}\/ of $G$ is  
the subgroup of $\Aut(G)$ consisting of all automorphisms 
inducing the identity on abelianization; that is, 
\begin{equation}
\label{eq:iag}
\IA(G)= \ker (\Aut(G) \to \Aut(G/[G,G]).
\end{equation}

\begin{example}
Let $F_n$ be the free group of rank $n$, 
and let $\Z^n$ be its abelianization.   Identify the automorphism 
group $\Aut(\Z^n)$ with the general linear group $\GL_n(\Z)$.   
As is well-known, the map $\Aut(F_n)\to \GL_n (\Z)$ 
which sends an automorphism to the induced map on the
abelianization is surjective.      
The Torelli group $\IA(F_n)=\ker (\Aut(F_n) \surj \GL_n(\Z))$ is 
classically denoted by $\IA_n$.   Magnus showed that this group is 
finitely generated.  
Clearly, $\IA_1=\{1\}$, while, as noted by Magnus, 
$\IA_2=\Inn(F_2)\cong F_2$.  
On the other hand,  Krsti\'{c} and McCool showed  
that $\IA_3$ admits no finite presentation.  It is still unknown 
whether $\IA_n$ admits a finite presentation for $n\ge 4$.
\end{example}

\begin{example}
Let $\Sigma_g$ be a Riemann surface of genus $g$, 
and let $\mathcal{I}_g=\IA(\pi_1(\Sigma_g))$ be the associated 
Torelli group.   For $g \le 1$, the group $\mathcal{I}_g$ is trivial, 
while for $g=2$, it is not finitely generated.  On the other hand, 
 it is known that $\mathcal{I}_g$ is finitely generated for $n\ge 3$.
\end{example}

As noted by Hain \cite{Hain97} in the case of the Torelli group of 
a Riemann surface and proved by Berceanu and Papadima 
\cite{Berceanu-Papadima09} in full generality, a stronger 
assumption on $G$ leads to a stronger conclusion on $\IA(G)$.

\begin{theorem}[\cite{Hain97, Berceanu-Papadima09}]
\label{thm:hbp}
Let $G$ be a finitely generated, residually nilpotent group, and suppose 
$\gr_k(G)$ is torsion-free for all $k\ge 1$.  Then the Torelli group 
$\IA(G)$ is residually torsion-free nilpotent. 
\end{theorem}

As shown by Magnus, all free group $F_n$ are residually torsion-free nilpotent. 
Hence, the Torelli groups $\IA(F_n)$ are  residually torsion-free nilpotent.  
Furthermore, all its subgroups, such as the pure braid group
$P_n$, the McCool group $wP_n$, and the 
upper McCool group $wP_n^+$ are also residually torsion-free nilpotent.
We refer to \cite{BB09, Marin12, SW-braids} for more details and references 
on this subject. 

\subsection{The RTFN property and Taylor expansions}
\label{subsec:rtfn-taylor}

The next result relates the RTFN property of a filtered-formal group 
to the injectivity of the corresponding Taylor expansion. 

\begin{prop}
\label{prop:fftaylor}
A finitely generated group $G$ has a faithful Taylor expansion 
if and only if $G$ is residually torsion-free nilpotent and filtered-formal. 
\end{prop}

\begin{proof}
By Corollary \ref{cor:TaylorFilteredFormal}, the group 
$G$ is filtered-formal if and only if there is 
a Taylor expansion $E\colon G\to \widehat{\gr}(\k G)$. 
In this case, by Propositions \ref{prop:Taylor} and \ref{prop:TaylorExp}, 
the map $\widehat{E}\colon \widehat{\k G}\to \widehat{\gr}(\k G)$ is an  
isomorphism of filtered Hopf algebras, which fits into the commuting diagram
\begin{equation}
\begin{gathered}
\xymatrix{
G \ar[r]^{\kappa}\ar[rd]_{E} &  \widehat{\k G} \ar[d]^{\widehat{E}}  \\
& \widehat{\gr}(\k G)\, .
}
\end{gathered}
\end{equation}
Hence, $E$ is injective if and only if $\kappa$ is injective.  
That is to say, the expansion $E$ is faithful if and only if 
the group $G$ is RTFN. 
\end{proof}

\begin{example}
\label{ex:braid-rtfn}
Consider the braid group $B_n$, with $n\ge 3$. Let us identify the complete Hopf algebra 
$\widehat{\gr}(\k B_n)$ with $\k\com{X}$, the power series ring over $\k$ in one variable.
The homomorphism $L\colon B_n\to \k\com{X}$ given by $L(\sigma_i)= \exp(X)$ 
is a Taylor expansion of the braid group, since 
$\log(\exp(\sigma_i)\exp(\sigma_j))=2X$ 
is a group-like element in $\k\com{X}$.
It is clear that this expansion is not faithful, since
$L([\sigma_1,\sigma_2])=0$ but $[\sigma_1,\sigma_2]\neq 1\in B_n$. 
In fact, it is known that the braid groups $B_n$ ($n\ge 3$) are 
not RTFN, see \cite{RolfsenZhu-98}.
\end{example}

\begin{ack}
We wish to thank Dror Bar-Natan for an inspiring conversation that led us to 
work on this project. 
\end{ack}

\newcommand{\arxiv}[1]
{\texttt{\href{http://arxiv.org/abs/#1}{arXiv:#1}}}
\newcommand{\arx}[1]
{\texttt{\href{http://arxiv.org/abs/#1}{arxiv:}}
\texttt{\href{http://arxiv.org/abs/#1}{#1}}}
\newcommand{\doi}[1]
{\texttt{\href{http://dx.doi.org/#1}{doi:#1}}}
\renewcommand*\MR[1]{%
\StrBefore{#1 }{ }[\temp]%
\href{http://www.ams.org/mathscinet-getitem?mr=\temp}{MR#1}}

\def\cprime{$'$}

\bibliographystyle{amsplain}

\end{document}